\documentclass[12pt,draft]{amsart}

\usepackage{amsmath,amssymb,amsthm,cite}

\newtheorem{problem}{Problem}

\newtheorem{theorem}{Theorem}

\newtheorem{lemma}{Lemma}

\newtheorem{corollary}{Corollary}

\DeclareMathOperator{\init}{Init}

\DeclareMathOperator{\term}{Term}

\DeclareMathOperator{\val}{val}

\DeclareMathOperator{\cl}{cl}

\begin{document}

\title[Semiring identities in $B_2$]{Semiring identities in the semigroup $B_2$}

\author{Vyacheslav Yu. Shaprynski\v i}\email{vshapr@yandex.ru}

\address{Ural Federal University, Institute of Natural Sciences and Mathematics, 620000, Ekaterinburg, Russia}


\thanks{This work was carried out with the financial support of the Ministry of Science and Higher Education of the Russian Federation (Project No. FEUZ-2026-0007)}


\begin{abstract}
The 5-element Brandt semigroup $B_2$ admits the structure of a naturally semilattice-ordered inverse semigroup, thus becoming an additively idempotent semiring with the operation of taking greatest lower bounds as the semiring addition. For this semiring we present a finite basis of identities and thus, by the previous results, complete the solution of Finite Basis Problem for combinatorial naturally semilattice-ordered inverse semigroups.
\end{abstract}

\maketitle

\section{Introduction and summary}

An \emph{inverse semigroup} is a semigroup that satisfies the condition
$$\forall x\exists!y\quad xyx=x\text{ and }yxy=y.$$
The element $y$ is called \emph{inverse} to $x$ and denoted by $x^{-1}$. The study of inverse semigroups is a classical branch of semigroup theory. We will use basic facts about inverse semigroups without direct reference, for the details see monographs~\cite{Lawson-99,Petrich-84}. Every inverse semigroup admits the \emph{natural partial order} relation defined by
$$x\le y\iff xx^{-1}y=x.$$
An inverse semigroup is called \emph{naturally semilattice-ordered} if it is a lower semilattice under this order relation.

As an example, one can mention rook monoids. The rook monoid $\mathcal R_n$ can be defined as the set of all injective partial transformations on the set $\{1,\dots,n\}$ under the operation of composition. This monoid is an inverse semigroup with inverse partial transformations being inverse elements of the semigroup and the natural order relation being the set-theoretical inclusion of partial transformations (as binary relations). The semigroup $\mathcal R_n$ is a naturally semilattice-ordered inverse semigroup: the set-theoretical intersection of partial transformations is their greatest lower bound.

The monoid $\mathcal R_n$ admits a natural $n\times n$-matrix representation. To each element $\varphi\in\mathcal R_n$, we assign the matrix $(a_{ij})$ defined by
$$a_{ij}=\begin{cases}
1\text{ if }\varphi(i)=j,\\
0\text{ otherwise.}
\end{cases}$$
It is easy to see that composition of partial transformations corresponds to multiplication of matrices, the operation of inversion corresponds to transposition of matrices, the natural order relation corresponds to the relation of elementwise comparison of matrices, and the semilattice operation corresponds to elementwise multiplication. A matrix represents an element of $\mathcal R_n$ if and only if it contains at most one unit in each row and in each column. This explains the term ``rook monoid'' since one can consider such matrices as arrangements of rooks on the $n\times n$-chessboard such that no two rooks attack each other. In the further text, we assume the matrix representation to be the definition of the monoid $\mathcal R_n$.

Two other naturally semilattice-ordered inverse semigroups are the 5-element Brandt semigroup $B_2$ and the 6-element Brandt monoid $B_2^1$. Both these semigroups can be considered inverse subsemigroups of $\mathcal R_2$. In matrix terms, they are defined by
$$B_2=\left\{\begin{pmatrix}
0&0\\
0&0
\end{pmatrix},
\begin{pmatrix}
1&0\\
0&0
\end{pmatrix},
\begin{pmatrix}
0&1\\
0&0
\end{pmatrix},
\begin{pmatrix}
0&0\\
1&0
\end{pmatrix},
\begin{pmatrix}
0&0\\
0&1
\end{pmatrix}\right\}$$
and
$$B_2^1=B_2\cup\left\{\begin{pmatrix}
1&0\\
0&1
\end{pmatrix}\right\}.$$

In every naturally semilattice-ordered inverse semigroup multiplication distributes over the operation of taking grestest lower bounds~\cite[Proposition~1.22]{Schein-73}, so it can be considered an additively idempotent semiring. Recall that an \emph{additively idempotent semiring} (\emph{ai-semiring} for short) is an algebra with two binary operations $(S,+,\cdot)$ such that the additive reduct $(S,+)$ is a semilattice, the multiplicative reduct $(S,\cdot)$ is a semigroup, and multiplication distributes over addition. In particular, one can consider the Finite Basis Problem in the semiring signature.

Recall that a system of identities is an \emph{identity basis} of an algebra $A$ if it holds in $A$ and implies all identities that hold in $A$. An algebra is \emph{finitely based} if it has a finite identity basis and \emph{nonfinitely based} otherwise. The \emph{Finite Basis Problem} (FBP for short) consists in determining which algebras are finitely based. 

It is  well-known that the monoid $B_2^1$ is nonfinitely based in the semigroup signature~\cite{Perkins-68}. The semigroup $B_2$ has the identity basis
 \begin{align}
\label{periodicity}
&x^2\approx x^3,\\
\label{shortening}
&xyxyx\approx xyx,\\
\label{commut}
&x^2y^2\approx y^2x^2.
\end{align}
For the history of the proof of this result see~\cite{Volkov-19}. The monoid $B_2^1$ remains nonfinitely based in the semiring signature~\cite{Jackson-22,Volkov-21}. It was pro\-bably this fact that motivated the following problem.

\begin{problem}[{\!\cite[Problem~7.7(3)]{Jackson-22}}]\label{inv sem} 
Which finite naturally semilattice-ordered inverse semigroups are finitely based, in either of the signatures $\{+,\cdot\}$ or $\{+,\cdot,0\}$?
\end{problem}

In the paper~\cite{Gusev-23}, rook monoids $\mathcal R_n$ are proved to be nonfinitely based for $n\ge2$ and Problem~\ref{inv sem} is investigated in the combinatorial case. Recall that a \emph{combinatorial semigroup} is a semigroup that contains no nontrivial subgroups. The main result of~\cite{Gusev-23} states that a finite combinatorial naturally semiattice-ordered inverse semigroup $S$ is non-finitely based as an ai-semiring whenever the inverse semigroup $B_2^1$ is contained in the variety of inverse semigroups generated by $S$. 

This result ``almost solves'' Problem~\ref{inv sem} in combinatorial case. Recall that two algebras are called \emph{equationally equivalent} if they satisfy the same identities. There are, up to equational equivalence, only three combinatorial inverse semigroups that do not satisfy the condition above: the trivial semigroup, the two-element semilattice, and the semigroup $B_2$. The first two of these three semigroups are obviously finitely based, so in combinatorial case it remains to solve the FBP for the semiring $B_2$. This problem was explicitly stated in~\cite{Jackson-22}, amongst with the FBP for the subsemiring
$$B_0=B_2\setminus\left\{\begin{pmatrix}
0&0\\
1&0
\end{pmatrix}\right\}.$$

\begin{problem}[{\!\cite[Problem~7.7(1)]{Jackson-22}}] Resolve the finite or nonfinite basabi\-lity of $B_0$ and $B_2$ as ai-semirings.
\end{problem}

A finite identity basis for $B_0$ was found in~\cite{Shaprynskii-24}. Note that $B_0$ is the only, up to isomorphism, 4-element subsemiring of $B_2$. All subsemi\-rings of $B_2$ with $\le3$ elements are finitely based~\cite{Zhao-20}. Thus all proper subsemirings of $B_2$ are finitely based.

The aim of the present paper is to prove that the semiring $B_2$ is finitely based, thereby completing the solution of the Finite Basis Problem for naturally semilattice-ordered inverse semigroups in the combinatorial case. We call semigroup terms \emph{words} and semiring terms \emph{polynomials}. For polynomials $\mathbf u$ and $\mathbf v$, the inequality $\mathbf u\le\mathbf v$ denotes the identity $\mathbf u+\mathbf v\approx\mathbf v$ which corresponds to the additive semilattice order relation in ai-semirings\footnote{There is a terminological inconsistency between the theory of inverse semigroups and the theory of ai-semirings. The definition of natural partial order supposes a naturally semilattice-ordered inverse semigroup to be a lower semilattice while additive semilattices of ai-semirings are usually considered upper semilattices. In the present paper we use the semiring-theoretical notation.}. The following theorem is the main result of the present paper.

\begin{theorem}
\label{main1}
The following identities constitute an identity basis for the semiring $B_2$ within the variety of all ai-semirings:
\begin{align}
\label{rook}
&x_2z_2\le x_1z_1+x_1z_2+x_2z_1,\\
\label{crossing}
&x_1yz_2\le x_1yz_1+x_2yz_2,
\end{align}
where $x_1,x_2,z_1,z_2$ can be empty.
\end{theorem}

In particular, Theorem~\ref{main1} states (somewhat counter-intuitively) that the identities (\ref{periodicity}), (\ref{shortening}), and (\ref{commut}) follow from (\ref{rook}) and (\ref{crossing}). Note that the identities (\ref{rook}) and (\ref{crossing}) have occurred in~\cite{Shaprynskii-24} as a part of the identity basis of $B_0$. The identity~(\ref{rook}) has been known since the work~\cite{Garvackii-71} where it was proved that this identity holds in all rook monoids.

The proof of Theorem~\ref{main1} uses an explicit description of identities in $B_2$. This description seems to be of independent interest. It is related to the ``edge closure'' technique used in~\cite{Reilly-08} to describe identities in $B_2$ in semigroup signature, however a different notation will be convenient in the semiring case. By $c(\mathbf p)$ we denote the \emph{content} of a polynomial $\mathbf p$, i.e. the set of all letters occurring in $\mathbf p$. We consider the following binary relations on the cartesian product $c(\mathbf p)\times\{1,2\}$:
\begin{align*}
&\rho_1(\mathbf p)=\{((x,1),(y,1))\mid\text{$x$ and $y$ are first letters of two words in $\mathbf p$}\};\\
&\rho_2(\mathbf p)=\{((x,2),(y,2))\mid\text{$x$ and $y$ are last letters of two words in $\mathbf p$}\};\\
&\rho_3(\mathbf p)=\{((x,2),(y,1))\mid\text{$xy$ is a subword of a word in $\mathbf p$}\}.
\end{align*}
We define $\rho(\mathbf p)$ as the equivalence relation on $c(\mathbf p)\times\{1,2\}$ generated by $\rho_1(\mathbf p)\cup\rho_2(\mathbf p)\cup\rho_3(\mathbf p)$. The \emph{initial \textup[terminal\textup]} $\rho(\mathbf p)$-\emph{class} is the $\rho(\mathbf p)$-class which contains the pair $(x,1)$ [$(x,2)$] where $x$ is the first [last] letter of a word in $\mathbf p$. We denote this class by $\init(\mathbf p)$ [$\term(\mathbf p)$].

\begin{theorem}
\label{main2} A semiring identity $\mathbf p\approx\mathbf q$ holds in $B_2$ if and only if it satisfies the following conditions
\begin{itemize}
\item[1)] $c(\mathbf p)=c(\mathbf q)$;
\item[2)] $\rho(\mathbf p)=\rho(\mathbf q)$;
\item[3)] $\init(\mathbf p)=\init(\mathbf q)$;
\item[4)] $\term(\mathbf p)=\term(\mathbf q)$.
\end{itemize}
\end{theorem}

Theorem~\ref{main2} easily implies that equation checking in $B_2$ is solvable in polynomial time. The proofs of Theorems~\ref{main1} and~\ref{main2} are given in the next section.

\section{Proofs of Theorems~\ref{main1} and~\ref{main2}}

We denote by $\Sigma$ the system of identities (\ref{rook}) and (\ref{crossing}). Let us fix notation for nonzero elements of $B_2$:
$$e_{11}=\begin{pmatrix}
1&0\\
0&0
\end{pmatrix},\ 
e_{12}=\begin{pmatrix}
0&1\\
0&0
\end{pmatrix},\ 
e_{21}=\begin{pmatrix}
0&0\\
1&0
\end{pmatrix},\ 
e_{22}=\begin{pmatrix}
0&0\\
0&1
\end{pmatrix}.$$
For an element $a=e_{ij}\in B_2$, we write $a[1]=i$ and $a[2]=j$. The numbers $a[1]$ and $a[2]$ are undefined for $a=0$. The following rules describe addition and multiplication in $B_2$
\begin{align}
\label{addition}
x+y&=\begin{cases}
x\text{ if }x=y,\\
0\text{ otherwise;}
\end{cases}\\
\label{multiplication}
xy&=\begin{cases}
e_{x[1],y[2]}\text{ if }x,y\neq0\text{ and }x[2]=y[1],\\
0\text{ otherwise.}
\end{cases}
\end{align}

\begin{lemma}
\label{sigma} The system $\Sigma$ holds in the semiring $B_2$.
\end{lemma}
\begin{proof}
(\ref{rook}) Follows from the results of~\cite{Garvackii-71} since the semiring $B_2$ is a subsemiring of $\mathcal R_2$.

(\ref{crossing}) Suppose we have assigned values in $B_2$ to all letters. We can denote these values by the same letters without ambiguity. In the nontrivial case, we can assume that the right hand side $x_1yz_1+x_2yz_2$ takes a non-zero value. There are four possible cases.

\emph{Case}~1. The letters $x_1$ and $z_2$ are non-empty. We have $x_1[2]=y[1]$ and $y[2]=z_2[1]$ by~(\ref{addition}) and~(\ref{multiplication}), whence $x_1yz_1+x_2yz_2=x_1yz_2$.

\emph{Case}~2. The letter $z_2$ is empty and the letter $x_1$ is not, so the identity is $x_1yz_1+x_2y\le x_1y$. We have $x_1[2]=y[1]$ by~(\ref{addition}) and~(\ref{multiplication}), whence $x_1yz_1+x_2y=x_1y$.

\emph{Case}~3. The letter $z_2$ is empty and the letter $x_1$ is not. This case is dual to the previous one.

\emph{Case}~4. The letters $x_1$ and $z_2$ are empty, so the identity is $yz_1+x_2y\le y$. The conclusion $yz_1+x_2y=y$ immediately follows from~(\ref{addition}) and~(\ref{multiplication}).
\end{proof}

We denote the $\rho(\mathbf p)$-class of a pair $(x,i)$ by $(x,i)\rho(\mathbf p)$. The following lemma constitutes the ``only if'' part of Theorem~\ref{main2}.

\begin{lemma}
\label{equivalence} If an identity $\mathbf p\approx\mathbf q$ holds in $B_2$ then it satisfies the conditions \textup{1)--4)} of Theorem~\textup{\ref{main2}}.
\end{lemma}
\begin{proof}
1) Suppose $c(\mathbf p)\not\subseteq c(\mathbf q)$. If we assign the value $e_{11}$ to all letters in $\mathbf q$ and the value $0$ to all letters in $c(\mathbf p)\setminus c(\mathbf q)$ then $\mathbf p$ takes the value $0$ and $\mathbf q$ takes the value $e_{11}$. Therefore, the identity $\mathbf p\approx\mathbf q$ fails in $B_2$.

2) Suppose $c(\mathbf p)=c(\mathbf q)$ and $\rho(\mathbf p)\not\subseteq\rho(\mathbf q)$. Therefore, 
$$\rho_1(\mathbf p)\cup\rho_2(\mathbf p)\cup\rho_3(\mathbf p)\not\subseteq\rho(\mathbf q).$$ 
Take a pair 
$$((x,i),(y,j))\in(\rho_1(\mathbf p)\cup\rho_2(\mathbf p)\cup\rho_3(\mathbf p))\setminus\rho(\mathbf q).$$ 
Since $((x,i),(y,j))\not\in\rho(\mathbf q)$, there exists a function 
\begin{equation}
\label{phi}
\varphi\colon(c(\mathbf q)\times\{1,2\})/\rho(\mathbf q)\rightarrow\{1,2\}
\end{equation}
which separates the classes $(x,i)\rho(\mathbf q)$ and $(y,j)\rho(\mathbf q)$. We associate with $\varphi$ the function $\val_{\varphi}\colon c(\mathbf q)\rightarrow B_2$ defined as follows: 
$$\val_{\varphi}(z)=e_{\varphi((z,1)\rho(\mathbf q)),\varphi((z,2)\rho(\mathbf q))}.$$
The function $\val_{\varphi}$ extends to the set of all polynomials over $c(\mathbf q)$. Take a word $\mathbf w=x_1x_2\dots x_k$ in $\mathbf q$. For $i=1,2,\dots,k-1$, we have $((x_i,2),(x_{i+1},1))\in\rho_3(\mathbf q)\subseteq\rho(\mathbf q)$. Therefore, 
$$\val_\varphi(x_i)[2]=\varphi((x_i,2)\rho(\mathbf q))=\varphi((x_{i+1},1)\rho(\mathbf q))=\val_\varphi(x_{i+1})[1].$$ 
Therefore, by~(\ref{multiplication}) we have $\val_{\varphi}(\mathbf w)\neq0$ and
$$\val_{\varphi}(\mathbf w)[1]=\val_{\varphi}(x_1)[1]=\varphi((x_1,1)\rho(\mathbf q))=\varphi(\init(\mathbf q)).$$
Similarly, $\val_{\varphi}(\mathbf w)[2]=\varphi(\term(\mathbf q))$. Hence 
$$\val_{\varphi}(\mathbf w)=e_{\varphi(\init(\mathbf q)),\varphi(\term(\mathbf q))}.$$
This holds for each word $\mathbf w$ in $\mathbf q$, so additive idempotency implies
$$\val_{\varphi}(\mathbf q)=e_{\varphi(\init(\mathbf q)),\varphi(\term(\mathbf q))}.$$ 
Furthermore, the condition $((x,i),(y,j))\in\rho_1(\mathbf p)\cup\rho_2(\mathbf p)\cup\rho_3(\mathbf p)$ admits three possible cases.

\emph{Case}~1: $((x,i),(y,j))\in\rho_3(\mathbf p)$. This means that $i=2$, $j=1$, and $xy$ is a subword of a word $\mathbf w$ in $\mathbf p$. We have
$$\val_{\varphi}(x)[2]=\varphi((x,2)\rho(\mathbf p))\neq\varphi((y,1)\rho(\mathbf p))=\val_{\varphi}(y)[1].$$ 
Therefore, by~(\ref{multiplication}) we have $\val_{\varphi}(xy)=0$. Since $xy$ is a subword in $\mathbf w$, this implies $\val_{\varphi}(\mathbf w)=0$ and $\val_{\varphi}(\mathbf p)=0$.

\emph{Case}~2: $((x,i),(y,j))\in\rho_1(\mathbf p)$. This means that $i=j=1$ and $x$ and $y$ are first letters of two words $\mathbf w_1,\mathbf w_2\in\mathbf p$ respectively. We have either $\val_{\varphi}(\mathbf w_1)=0$, or $\val_{\varphi}(\mathbf w_2)=0$, or 
$$\val_{\varphi}(\mathbf w_1)[1]=\varphi((x,1)\rho(\mathbf q))\neq\varphi((y,1)\rho(\mathbf q))=\val_{\varphi}(\mathbf w_2)[1].$$ 
In either case, $\val_{\varphi}(\mathbf w_1+\mathbf w_2)=0$, whence $\val_{\varphi}(\mathbf p)=0$.

\emph{Case}~3: $((x,i),(y,j))\in\rho_2(\mathbf p)$. This case is dual to the previous one.

In all three cases, $\val_{\varphi}(\mathbf p)=0\neq\val_{\varphi}(\mathbf q)$, so the identity $\mathbf p\approx\mathbf q$ fails in $B_2$.

3) Suppose $c(\mathbf p)=c(\mathbf q)$, $\rho(\mathbf p)=\rho(\mathbf q)$, and $\init(\mathbf p)\neq\init(\mathbf q)$. There exists a function~(\ref{phi}) with $\varphi(\init(\mathbf p))\neq\varphi(\init(\mathbf q))$. The evaluation function $\val_{\varphi}$ is defined in the same way as in 2). Repeating the argument in 2), we conclude that $\val_{\varphi}(\mathbf p)=e_{\varphi(\init(\mathbf p)),\varphi(\term(\mathbf p))}$ and $\val_{\varphi}(\mathbf q)=e_{\varphi(\init(\mathbf q)),\varphi(\term(\mathbf q))}$. Since $\varphi(\init(\mathbf p))\neq\varphi(\init(\mathbf q))$, we conclude that $\val_{\varphi}(\mathbf p)\neq\val_{\varphi}(\mathbf q)$, so the identity $\mathbf p\approx\mathbf q$ fails in $B_2$.

4) This item is dual to the previous one.
\end{proof}

For convenience of references, note the following obvious principle:
\begin{equation}
\label{adding inequalities}\mathbf p\le\mathbf r\text{ and }\mathbf q\le\mathbf  r\text{ imply }\mathbf p+\mathbf q\le\mathbf r.
\end{equation}

We will need two other observations. We collect them into the following lemma.

\begin{lemma}
\label{easy}
\textup{1)} If $((x,i),(y,j))\in\rho(\mathbf p)$ and $(x,i)\neq(y,j)$ then there exists a sequence
\begin{equation}
\label{deduction}(z_1,n_1),(z_2,n_2),\dots,(z_k,n_k)
\end{equation}
such that $(z_1,n_1)=(x,i)$, $(z_k,n_k)=(y,j)$, and
$$((z_l,n_l),(z_{l+1},n_{l+1}))\in\rho_1(\mathbf p)\cup\rho_2(\mathbf p)\cup\rho_3(\mathbf p)\cup\rho_3(\mathbf p)^{-1}$$
for $l=1,2,\dots,k-1$.

\textup{2)} For $l=1,2,\dots,k-1$ we have
\begin{align*}
&((z_l,n_l),(z_{l+1},n_{l+1}))\in\rho_1(\mathbf p)\iff n_l=1\text{ and }n_{l+1}=1,\\
&((z_l,n_l),(z_{l+1},n_{l+1}))\in\rho_2(\mathbf p)\iff n_l=2\text{ and }n_{l+1}=2,\\
&((z_l,n_l),(z_{l+1},n_{l+1}))\in\rho_3(\mathbf p)\iff n_l=2\text{ and }n_{l+1}=1,\\
&((z_l,n_l),(z_{l+1},n_{l+1}))\in\rho_3(\mathbf p)^{-1}\iff n_l=1\text{ and }n_{l+1}=2.
\end{align*}
\end{lemma}
\begin{proof}
1) The equivalence relation generated by a relation $\alpha$ can be described as $\Delta\cup\cl(\alpha\cup\alpha^{-1})$ where $\Delta$ is the equality relation and $\cl$ is the transitive closure operator. If $(x,i)=(y,j)$ then we can take the trivial sequence~(\ref{deduction}) with $k=1$. Suppose $(x,i)\neq(y,j)$. Since the relations $\rho_1(\mathbf p)$ and $\rho_2(\mathbf p)$ are symmetric by definition, we have
\begin{align*}
&((x,i),(y,j))\in\cl((\rho_1(\mathbf p)\cup\rho_2(\mathbf p)\cup\rho_3(\mathbf p))\cup(\rho_1(\mathbf p)\cup\rho_2(\mathbf p)\cup\rho_3(\mathbf p))^{-1})\\
&=\cl(\rho_1(\mathbf p)\cup\rho_2(\mathbf p)\cup\rho_3(\mathbf p)\cup\rho_1(\mathbf p)^{-1}\cup\rho_2(\mathbf p)^{-1}\cup\rho_3(\mathbf p)^{-1})\\
&=\cl(\rho_1(\mathbf p)\cup\rho_2(\mathbf p)\cup\rho_3(\mathbf p)\cup\rho_3(\mathbf p)^{-1}).
\end{align*}

2) This statement immediately follows from the definitions of $\rho_1(\mathbf p)$, $\rho_2(\mathbf p)$, and $\rho_3(\mathbf p)$.
\end{proof}

We write $\mathbf p\approx_{\Sigma}\mathbf q$ [$\mathbf p\le_{\Sigma}\mathbf q$, $\mathbf p\ge_{\Sigma}\mathbf q$] if the identity $\mathbf p\approx\mathbf q$ [$\mathbf p\le\mathbf q$, $\mathbf p\ge\mathbf q$] follows from $\Sigma$.

\begin{lemma}
\label{first letters} Let $\mathbf p$ be a polynomial and $x\in c(\mathbf p)$. If $(x,1)\in\init(\mathbf p)$ then there exists a word $x\mathbf w\le_{\Sigma}\mathbf p$.
\end{lemma}
\begin{proof}
Let $y$ be the first letter of a word in $\mathbf p$. We have $(x,1)\in\init(\mathbf p)$ and $(y,1)\in\init(\mathbf p)$, so $((y,1),(x,1))\in\rho(\mathbf p)$. Using Lemma~\ref{easy}.1), consider the corresponding sequence~(\ref{deduction}). We use induction on the length $k$ of this sequence. 

\textbf{Induction base:} $k=1$. In this case $x=y$, so there exists a word $x\mathbf w$ in $\mathbf p$. In particular, $x\mathbf w\le_{\Sigma}\mathbf p$.

\textbf{Induction step.} Let $(z_t,n_t)$ be the last element of the sequence~(\ref{deduction}) such that $t\le k-1$ and $n_t=1$. There are three possible cases.

\emph{Case}~1: $t=k-1$. By Lemma~\ref{easy}.2), we have 
$$((z_{k-1},1),(x,1))\in\rho_1(\mathbf p).$$ 
By definition of $\rho_1(\mathbf p)$, there exists a word $x\mathbf w$ in $\mathbf p$.

\emph{Case}~2: $t=k-2$. The last three elements of the sequence~(\ref{deduction}) are $(z_{k-2},1),(z_{k-1},2),(x,1)$. By Lemma~\ref{easy}.2), we have 
$$((z_{k-2},1),(z_{k-1},2))\in\rho_3(\mathbf p)^{-1}$$ 
and 
$$((z_{k-1},2),(x,1))\in\rho_3(\mathbf p).$$
By definition of $\rho_3(\mathbf p)$, there exist words $\mathbf u_1z_{k-1}z_{k-2}\mathbf u_2$, $\mathbf v_1z_{k-1}x\mathbf v_2$ in $\mathbf p$. By induction hypothesis, we have $z_{k-2}\mathbf w'\le_{\Sigma}\mathbf p$ for some $\mathbf w'$. Hence
\begin{align*}
\mathbf p&\ge_{\Sigma}\mathbf u_1z_{k-1}z_{k-2}\mathbf u_2+\mathbf v_1z_{k-1}x\mathbf v_2+z_{k-2}\mathbf w'&&\text{by~(\ref{adding inequalities})}\\
&\ge_{\Sigma}\mathbf u_1z_{k-1}z_{k-2}\mathbf u_2+\mathbf u_1z_{k-1}x\mathbf v_2+z_{k-2}\mathbf w'&&\text{by~(\ref{crossing}) and~(\ref{adding inequalities})}\\
&\ge_{\Sigma}\mathbf u_1z_{k-1}z_{k-2}\mathbf u_2+\mathbf u_1z_{k-1}x\mathbf v_2+z_{k-2}\mathbf u_2&&\text{by~(\ref{crossing}) and~(\ref{adding inequalities})}\\
&\ge_{\Sigma}x\mathbf v_2.&&\text{by~(\ref{rook})}
\end{align*}

\emph{Case}~3: $t\le k-3$. Consider the elements $(z_t,1)$, $(z_{t+1},2)$, $(z_{k-1},2)$, $(x,1)$ of the sequence~(\ref{deduction}). By Lemma~\ref{easy}.2), we have
$$(z_{t+1},2)\,\rho_2(\mathbf p)\,(z_{t+2},2)\,\rho_2(\mathbf p)\cdots\rho_2(\mathbf p)\,(z_{k-1},2).$$
Therefore, there exist words $\mathbf w_1z_{t+1}$ and $\mathbf w_2z_{k-1}$ in $\mathbf p$. By induction hypothesis, there exists a word $z_t\mathbf w_3\le_{\Sigma}\mathbf p$. By Lemma~\ref{easy}.2), we have 
$$((z_t,1),(z_{t+1},2))\in\rho_3(\mathbf p)^{-1}$$ 
and 
$$((z_{k-1},2),(x,1))\in\rho_3(\mathbf p).$$ 
Therefore, there exist words $\mathbf u_1z_{t+1}z_t\mathbf u_2$ and $\mathbf v_1z_{k-1}x\mathbf v_2$ in $\mathbf p$. We have
\begin{align*}
\mathbf p&\ge_{\Sigma}\mathbf w_1z_{t+1}+\mathbf u_1z_{t+1}z_t\mathbf u_2+\mathbf w_2z_{k-1}&&\text{by~(\ref{adding inequalities})}\\
&\ge_{\Sigma}\mathbf u_1z_{t+1}+\mathbf u_1z_{t+1}z_t\mathbf u_2+\mathbf w_2z_{k-1}&&\text{by~(\ref{crossing}) and~(\ref{adding inequalities})}\\
&\ge_{\Sigma}\mathbf w_2z_{k-1}z_t\mathbf u_2,&&\text{by~(\ref{rook})}
\end{align*}
whence
\begin{align*}
\mathbf p&\ge_{\Sigma}\mathbf w_2z_{k-1}z_t\mathbf u_2+\mathbf v_1z_{k-1}x\mathbf v_2+z_t\mathbf w_3&&\text{by~(\ref{adding inequalities})}\\
&\ge_{\Sigma}\mathbf v_1z_{k-1}z_t\mathbf u_2+\mathbf v_1z_{k-1}x\mathbf v_2+z_t\mathbf w_3&&\text{by~(\ref{crossing}) and~(\ref{adding inequalities})}\\
&\ge_{\Sigma}\mathbf v_1z_{k-1}z_t\mathbf u_2+\mathbf v_1z_{k-1}x\mathbf v_2+z_t\mathbf u_2&&\text{by~(\ref{crossing}) and~(\ref{adding inequalities})}\\
&\ge_{\Sigma}x\mathbf v_2.&&\text{by~(\ref{rook})}
\end{align*}
\end{proof}

\begin{corollary}
\label{last letters} Let $\mathbf p$ be a polynomial and $x\in c(\mathbf p)$. If $(x,2)\in\term(\mathbf p)$ then there exists a word $\mathbf wx\le_{\Sigma}\mathbf p$.
\end{corollary}
\begin{proof}
This statement is dual to Lemma~\ref{first letters}.
\end{proof}

\begin{lemma} 
\label{inner letters} Let $\mathbf p$ be a polynomial and $x,y\in c(\mathbf p)$. If $((x,2),(y,1))\in\rho(\mathbf p)$ then there exists a word $\mathbf w_1xy\mathbf w_2\le_{\Sigma}\mathbf p$. 
\end{lemma}
\begin{proof}
Using Lemma~\ref{easy}.1), consider the sequence~(\ref{deduction}). We use induction on $k$.

\textbf{Induction base:} $k=2$. By Lemma~\ref{easy}, we have 
$$((x,2),(y,1))\in\rho_3(\mathbf p).$$ 
Hence there exists a word $\mathbf w_1xy\mathbf w_2$ in $\mathbf p$.

\textbf{Induction step.} Consider three cases.

\textbf{Case~1:} $n_{k-1}=1$. By induction hypothesis, there exists a word $\mathbf u_1xz_{k-1}\mathbf u_2\le_{\Sigma}\mathbf p$. By Lemma~\ref{easy}.2), we have 
$$((z_{k-1},1),(y,1))\in\rho_2(\mathbf p),$$
whence there exist words $z_{k-1}\mathbf v_1$ and $y\mathbf v_2$ in $\mathbf p$. Hence
\begin{align*}
\mathbf p&\ge_{\Sigma}z_{k-1}\mathbf v_1+y\mathbf v_2+\mathbf u_1xz_{k-1}\mathbf u_2&&\text{by~(\ref{adding inequalities})}\\
&\ge_{\Sigma}z_{k-1}\mathbf u_2+y\mathbf v_2+\mathbf u_1xz_{k-1}\mathbf u_2&&\text{by~(\ref{crossing}) and~(\ref{adding inequalities})}\\
&\ge_{\Sigma}\mathbf u_1xy\mathbf v_2.&&\text{by~(\ref{rook})}
\end{align*}

\textbf{Case~2:} $n_{k-2}=n_{k-1}=2$. By Lemma~\ref{easy}.2), 
$$((z_{k-2},2),(z_{k-1},2))\in\rho_2(\mathbf p),$$ 
whence $(z_i,n_i)\in\term(\mathbf p)$ for $i=1,\dots,k$. By Corollary~\ref{last letters}, there exist words $\mathbf ux\le_{\Sigma}\mathbf p$, $\mathbf vz_{k-1}\le_{\Sigma}\mathbf p$. By Lemma~\ref{easy}.2), 
$$((z_{k-1},2),(y,1))\in\rho_3(\mathbf p),$$ 
whence there exists a word $\mathbf w_1z_{k-1}y\mathbf w_2$ in $\mathbf p$. We have
\begin{align*}
\mathbf p&\ge_{\Sigma}\mathbf vz_{k-1}+\mathbf w_1z_{k-1}y\mathbf w_2+\mathbf ux&&\text{by~(\ref{adding inequalities})}\\
&\ge_{\Sigma}\mathbf vz_{k-1}+\mathbf vz_{k-1}y\mathbf w_2+\mathbf ux&&\text{by~(\ref{crossing}) and~(\ref{adding inequalities})}\\
&\ge_{\Sigma}\mathbf uxy\mathbf w_2.&&\text{by~(\ref{rook})}
\end{align*}

\textbf{Case~3:} $n_{k-2}=1$ and $n_{k-1}=2$. By induction hypothesis, there exists a word $\mathbf u_1xz_{k-2}\mathbf u_2\le_{\Sigma}\mathbf p$. By Lemma~\ref{easy}.2), 
$$((z_{k-2},1),(z_{k-1},2))\in\rho_3(\mathbf p)^{-1}$$ 
and 
$$((z_{k-1},2),(y,1))\in\rho_3(\mathbf p).$$ 
Therefore, there exist words $\mathbf v_1z_{k-1}z_{k-2}\mathbf v_2$ and $\mathbf w_1z_{k-1}y\mathbf w_2$ in $\mathbf p$. We have
\begin{align*}
\mathbf p&\ge_{\Sigma}\mathbf v_1z_{k-1}z_{k-2}\mathbf v_2+\mathbf w_1z_{k-1}y\mathbf w_2+\mathbf u_1xz_{k-2}\mathbf u_2&&\text{by~(\ref{adding inequalities})}\\
&\ge_{\Sigma}\mathbf v_1z_{k-1}z_{k-2}\mathbf u_2+\mathbf w_1z_{k-1}y\mathbf w_2+\mathbf u_1xz_{k-2}\mathbf u_2&&\text{by~(\ref{crossing}) and~(\ref{adding inequalities})}\\
&\ge_{\Sigma}\mathbf v_1z_{k-1}z_{k-2}\mathbf u_2+\mathbf v_1z_{k-1}y\mathbf w_2+\mathbf u_1xz_{k-2}\mathbf u_2&&\text{by~(\ref{crossing}) and~(\ref{adding inequalities})}\\
&\ge_{\Sigma}\mathbf u_1xy\mathbf w_2.&&\text{by~(\ref{rook})}
\end{align*}
\end{proof}

\begin{lemma}
\label{final} Each identity which satisfies the conditions \textup{1)}--\textup{4)} of Theorem~\textup{\ref{main2}} follows from $\Sigma$.
\end{lemma}
\begin{proof}
Let $\mathbf p\approx\mathbf q$ be an identity satisfying the conditions. Take a word $\mathbf w$ in $\mathbf p$. Take a prefix $\mathbf w_1$ of $\mathbf w$, so that $\mathbf w=\mathbf w_1\mathbf w_2$. We will prove that there exists a word $\mathbf w_1\mathbf w'\le_{\Sigma}\mathbf q$. We use induction on the length of $\mathbf w_1$.

\textbf{Induction base:} $\mathbf w_1=x$. The word $\mathbf w=x\mathbf w_2$ is in $\mathbf p$, so $(x,1)\in\init(\mathbf p)=\init(\mathbf q)$. Now the statement directly follows from Lemma~\ref{first letters}.

\textbf{Induction step:} $\mathbf w_1=\mathbf w'_1x$ and $\mathbf q\ge_{\Sigma}\mathbf w'_1\mathbf w'$ for some $\mathbf w'$. Let $y$ be the last letter of $\mathbf w'_1$, so that $\mathbf w'_1=\mathbf w''_1y$ and $\mathbf q\ge_{\Sigma}\mathbf w''_1y\mathbf w'$. Since $\mathbf w=\mathbf w_1\mathbf w_2=\mathbf w''_1yx\mathbf w_2$ is in $\mathbf p$, we have $((y,2),(x,1))\in\rho(\mathbf p)=\rho(\mathbf q)$. By Lemma~\ref{inner letters}, there exists a word $\mathbf u_1yx\mathbf u_2\le_{\Sigma}\mathbf q$. We have
\begin{align*}
\mathbf q&\ge_{\Sigma}\mathbf w''_1y\mathbf w'+\mathbf u_1yx\mathbf u_2&&\text{by~(\ref{adding inequalities})}\\
&\ge_{\Sigma}\mathbf w''_1yx\mathbf u_2.&&\text{by~(\ref{crossing})}
\end{align*}

We have finished the proof by induction. Applying the statement for the case $\mathbf w_1=\mathbf w$, we have $\mathbf q\ge_{\Sigma}\mathbf w\mathbf w'$. Let $x$ be the last letter in $\mathbf w$, so that $\mathbf w=\mathbf w''x$. Since $\mathbf w$ is a word in $\mathbf p$, we have $(x,2)\in\term(\mathbf p)=\term(\mathbf q)$. By Corollary~\ref{last letters}, there exists a word $\mathbf ux\le_{\Sigma}\mathbf q$. We have
\begin{align*}
\mathbf q&\ge_{\Sigma}\mathbf w''x\mathbf w'+\mathbf ux&&\text{by~(\ref{adding inequalities})}\\
&\ge_{\Sigma}\mathbf w''x.&&\text{by~(\ref{crossing})}
\end{align*}
We see that $\mathbf q\ge_{\Sigma}\mathbf w$. This holds for every word $\mathbf w$ in $\mathbf p$, whence $\mathbf p\le_{\Sigma}\mathbf q$ by~(\ref{adding inequalities}). Similarly, $\mathbf q\le_{\Sigma}\mathbf p$, whence $\mathbf p\approx_{\Sigma}\mathbf q$.
\end{proof}

\emph{Proof of Theorems}~\ref{main1} \emph{and}~\ref{main2}. Directly follows from Lemmas~\ref{sigma}, \ref{equivalence}, and \ref{final}.


\begin{thebibliography}{99}

\bibitem{Garvackii-71} Garvac$'$ki\u\i\ V.S.: $\cap$-semigroups of transformations. Theory of Semigroups and its Applications 2, Izdat. Saratov. Uni., Saratov, 2--13 (1971) (in Russian)

\bibitem{Gusev-23} Gusev S.V., Volkov M.V.: Semiring identities of finite inverse semigroups. Semigroup Forum 106, 403--420 (2023)

\bibitem{Jackson-22} Jackson M., Ren M., Zhao X.: Nonfinitely based ai-semirings with finitely based semigroup reducts. J. Algebra 611, 211--245 (2022)

\bibitem{Lawson-99} Lawson M.V.: Inverse Semigroups. The Theory of Partial Symmetries. World Scientific, Singapore (1999)

\bibitem{Perkins-68} Perkins P.: Bases for equational theories of semigroups. J. Algebra 11, 298--314 (1968)

\bibitem{Petrich-84} Petrich M.: Inverse Semigroups. John Wiley \& Sons, New York (1984)

\bibitem{Reilly-08} Reilly N.R.: The interval $[\mathbf B_2,\mathbf{NB}_2]$ in the lattice of Rees-Sushkevich varieties. Algebra Universalis 59, no. 3--4, 345--363 (2008)

\bibitem{Schein-73} Schein B.M.: Completions, thanslational hulls and ideal extensions of inverse semigroups. Czechoslovak Math. J. 23(4), 575--610 (1973)

\bibitem{Shaprynskii-24} Shaprynski\v\i\ V.Yu.: Semiring identities in the semigroup $B_0$. Semigroup Forum 109, 693--705 (2024) 

\bibitem{Volkov-19} Volkov M.V.: Identities in Brandt semigroups, revisited. Ural Math. J. 5(2), 80--93 (2019)

\bibitem{Volkov-21} Volkov M.V.: Semiring identities of the Brandt monoid. Algebra Universalis 82, 42 (2021)

\bibitem{Zhao-20} Zhao X.Z., Ren M.M., Crvenkovi\'c S., Shao Y., \DJ api\'c P.: The variety generated by an ai-semiring of order three. Ural Math. J. 6(2), 117--132 (2020)
\end{thebibliography}
\end{document}